\documentclass[11pt,a4paper]{article}

\usepackage{amsmath,amsfonts,amssymb,latexsym,graphics,epsfig,url}
\usepackage{xcolor}
\usepackage{amsthm}
\usepackage[english]{babel}
\usepackage{mathdots}
\usepackage{graphicx}
\usepackage{soul} %subrrayar
\usepackage[utf8]{inputenc}
\usepackage{diagbox}
\usepackage[makeroom]{cancel}
\usepackage{multicol}
\usepackage{float}              
\usepackage{algorithm}
\usepackage{algorithmic}

\newtheorem{theorem}{Theorem}[section]
\newtheorem{proposition}[theorem]{Proposition}

\newtheorem{conjecture}[theorem]{Conjecture}

\def\Z{\ns Z}

\def\vec0{\mbox{\boldmath $0$}}

\def\Z{\ns{Z}}

\def\1{\mbox{\boldmath $1$}}

\def\Z{\mathbb Z}

\begin{document}
	
\title{A note on three-quarters circulant digraphs 
\thanks{This research has been supported by
AGAUR from the Catalan Government under project 2021SGR00434 and MICINN from the Spanish Government under project PID2020-115442RB-I00.
% The research of M. A. Fiol was also supported by a grant from the  Universitat Polit\`ecnica de Catalunya with references AGRUPS-2022 and AGRUPS-2023.
}
}
	\author{C. Dalf\'o$^a$, M. A. Fiol$^b$,   M. A. Reyes$^a$\\
		\\
		{\small $^a$Dept. de Matem\`atica, Universitat de Lleida, Igualada (Barcelona), Catalonia}\\
		{\small {\tt \{monicaandrea.reyes,cristina.dalfo\}@udl.cat}}\\
		{\small $^{b}$Dept. de Matem\`atiques, Universitat Polit\`ecnica de Catalunya, Barcelona, Catalonia} \\
		{\small Barcelona Graduate School of Mathematics} \\
		{\small  Institut de Matem\`atiques de la UPC-BarcelonaTech (IMTech)}\\
		{\small {\tt miguel.angel.fiol@upc.edu} }\\
	}

\date{}
\maketitle
	
\begin{abstract}
% In this paper, we introduce a new infinite family of circulant digraphs with two steps, $a$ and $b$, such that all paths must be formed either by steps \blue{$(+a,+b)$}, \blue{$(+a,-b)$}, or \blue{$(-a,+b)$}.
% To represent such digraphs, which we call three-quarters digraphs (only three of the four directions in the plane are admissible), we use plane tessellations. 
% This allows us to obtain good results
% about their minimum diameter for every number of vertices, as well as their maximum number of vertices for every given diameter, in the context of the degree/diameter and degree/number of vertices problems.

We introduce and study a new family of circulant digraphs associated with the cyclic group $\Z_N$, obtained by restricting admissible combinations of two generators $a$ and $b$ to three coordinate sectors. The resulting distance-like function differs from the standard directed distance in circulant digraphs and gives rise to new geometric and combinatorial phenomena. Using planar lattice representations and periodic tessellations, we analyze the growth of reachable sets and derive Moore-type upper bounds for the corresponding order/diameter problem.

We construct explicit infinite families of three-quarters circulant structures with the prescribed diameter and provide lattice-based methods for determining admissible generator pairs. Separate constructions are obtained for even and odd diameters. In addition, computational experiments for small and moderate orders suggest improved families for even diameters and motivate a conjectural asymptotic formula for the maximum attainable order. The paper highlights the interplay between constrained lattice representations, periodic tilings, and extremal problems for circulant networks.
\end{abstract}

	
%%%%%%%%%%%%%%%%%%%%%%%%%%%%%%%%%%%%%%%%%%%%%%%%%%%%%%%%%%%%%%%%%%%%%%%%%%%%%%%%%%%%%%%%%%%%%%%%%%%%%%%%%%%%%%%%%%%%%%%%%%%%%%%%%%%%

\section{Introduction}

Circulant digraphs form a central class of Cayley digraphs on cyclic groups and have been extensively studied due to their rich algebraic structure and wide applicability in areas such as communication networks, parallel architectures, and combinatorial optimization. Their vertex-transitivity and regularity make them natural candidates for addressing extremal problems, most notably the \emph{degree/diameter problem}, which seeks the largest possible number of vertices in a graph or digraph with given maximum degree and diameter (see Bermond, Delorme, and Quisquater~\cite{bdq86}, and the comprehensive survey by Miller and \v{S}ir\'a\v{n} \cite{ms05}).

A classical approach to studying distance-related properties in such graphs relies on embedding them into geometric structures. In particular, the representation of circulant digraphs by means of plane tessellations has proved to be a powerful tool to analyze distances, diameters, and shortest paths. This methodology, which can be traced back to the work of Wong and Coppersmith~\cite{wc74}, was further developed in a series of papers by Morillo, Fiol, and F\`abrega \cite{mff85}, Morillo, Comellas, and Fiol \cite{mcf87}, and Yebra, Fiol, Morillo, and Alegre~\cite{yfma85}. In these papers, optimal constructions for families of graphs associated with planar tilings were obtained. These techniques provide both an intuitive geometric interpretation and an effective combinatorial framework for obtaining extremal results.

Let $CD(N,a,b)$ denote the circulant digraph with vertex set $\mathbb{Z}_N$ and generating set $\{a,b\}$. Such digraphs are regular of degree two and vertex-transitive, and their (strongly)connectivity is determined by the condition $\gcd(N,a,b)=1$. Traditionally, distances in $CD(N,a,b)$ are defined by directed paths using only the generators $a$ and $b$. However, this standard notion does not exhaust all possible combinatorial structures that can arise from such generating sets.

In this paper, we introduce and investigate a new infinite family of circulant digraphs, which we call \emph{three-quarters digraphs}. These digraphs are defined by allowing paths that use %restricted
combinations of steps of the form $(+a,+b)$, $(+a,-b)$, and $(-a,+b)$. This modification induces a non-standard distance function that differs significantly from the classical directed distance. As a consequence, the resulting metric structure exhibits new geometric and combinatorial features that are not present in the usual circulant setting.

One of the main advantages of this model is that the vertices reachable from a fixed origin can be arranged in a planar pattern that resembles a discrete tessellation. In this representation, the number of vertices at distance exactly $\ell$ grows linearly as $3\ell+1$, leading to cumulative counts of the form $1,5,12,22,\dots$. This observation enables us to derive upper bounds on the number of vertices attainable for a given diameter $k$. However, unlike in previously studied cases, these optimal patterns do not always extend to valid digraphs, since the corresponding tiles fail to tessellate the plane periodically. This phenomenon highlights a fundamental obstruction linking combinatorial optimality with geometric feasibility.

To overcome this difficulty, we adopt a lattice-based approach inspired by earlier work on plane tessellations~\cite{mff85,mcf87}. By identifying suitable translation vectors that generate periodic tilings, we derive systems of linear equations whose solutions yield admissible pairs of generators $a$ and $b$. This approach allows us to construct infinite families of three-quarters digraphs and to distinguish between the cases of even and odd diameter. In particular, explicit formulas for the parameters are obtained by solving these systems over the integers, thereby providing constructive solutions to the degree/diameter problem within this framework.

The study of three-quarters digraphs thus combines algebraic, geometric, and combinatorial techniques. It contributes to the broader understanding of how modifications to path structures affect metric properties in Cayley digraphs and reveals new connections among planar tessellations, lattice theory, and extremal graph constructions. Moreover, these results complement previous work on circulant and related digraphs (see, for instance, 
Wong and Coppersmith~\cite{wc74}, %Fiol~\cite{f87},
Morillo, Fiol,  and  F\`abrega \cite{mff85}
and Esqué, Aguiló and Fiol \cite{eaf93}
% , and 
% Dalf\'o, Fiol, Miller, Ryan, and \v{S}ir\'a\v{n}~\cite{dfmrs19}
),
and may have further implications for the design of efficient interconnection networks and routing schemes.

The paper is organized as follows. In Section~2, we formally define three-quarters digraphs and describe their planar representation. We then analyze their distance structure and derive bounds on their diameter and order. Finally, we present constructive methods and computational results that illustrate the theoretical findings.

% Representing some graphs and digraphs using plane tessellations has proved very useful for studying their distance-related parameters. These include, for instance, the diameter, the mean distance, existence of shortest paths
% between vertices, etc.
% To our knowledge, the first reference to the use of plane tessellations in the study of certain regular digraphs of degree two appears in the work of Wong and Coppersmith \cite{wc74}.

% This paper is structured as follows. In the next section, we present... 

%%%%%%%%%%%%%%%%%%%%%%%%%%%%%%%%%%%%%%%%%%%%%%%%%%%%%%%%%%%%%%%%%%%%

\section{Three-quarters digraphs}
\label{sec:CR}

\begin{figure}[t]
    \centering
    \includegraphics[width=6cm]{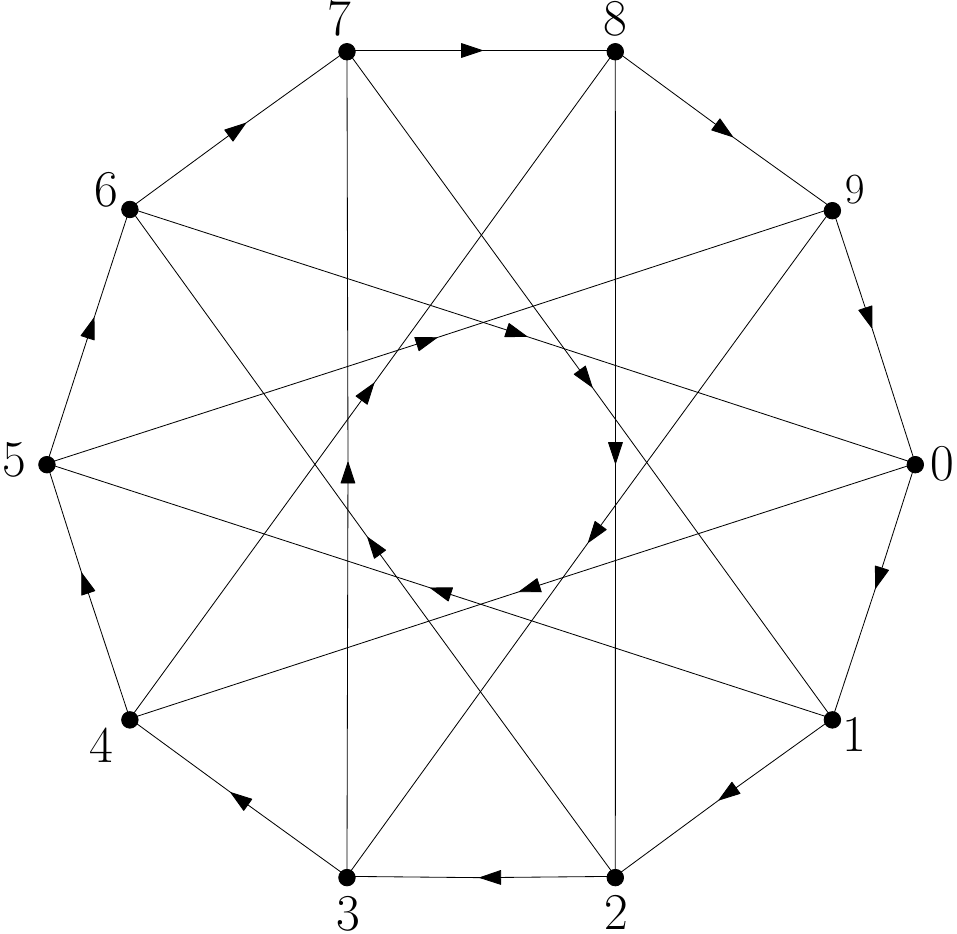}
    \caption{The circulant digraph $CD(10,1,4)\cong TQ(10,1,4)$.}
    \label{fig:circulant}
\end{figure}

A digraph is defined as an ordered pair $G=(V,A)$,
where $V$ is a non-empty set of objects called vertices, and $A$ is a set of ordered pairs of vertices called arcs (or directed edges).
A directed path in a digraph is a sequence of vertices 
 where every adjacent pair is connected by a directed edge pointing from the first vertex to the second.

The \textit{circulant digraph} $CD(N,a,b)$ is the Cayley digraph on $\Z_N$ with generating set $a,b$. Then, its $N$ vertices are identified with the integers modulo $N$, and each vertex $i$ is connected to the vertices $i+a \ (\textrm{mod }N)$ and $i+b \ (\textrm{mod }N)$ for the integers (called {\em steps)} $a$ and $b$, with $1\le a,b\le N-1$. See Figure \ref{fig:circulant} for the circulant digraph $CD(10,1,4)$.

Some simple properties of these digraphs are the following:
\begin{itemize}
\item[{\bf P1.}] 
The digraphs $CD(N,a,b)$ are regular of degree 2.
\item[{\bf P2.}]
A necessary and sufficient condition for $CD(N,a,b)$ to be strongly connected is 
$\gcd(N,a,b)=1$.
\item[{\bf P3.}] 
The digraphs $CD(N,a,b)$ are vertex-symmetric.
Thus, the study of its diameter can be done from any arbitrary vertex, say, 0.
\end{itemize}

\begin{figure}[t]
    \centering
    \includegraphics[width=8cm]{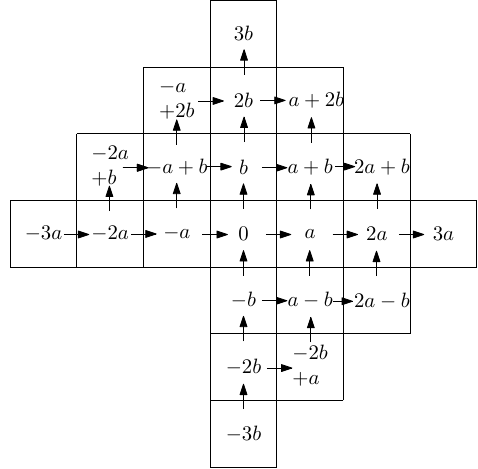}
    \caption{Adjacencies and planar pattern of the vertices in the digraph $CD(N,a,b)$ (with vertices at distance at most three from vertex $0$).}
    \label{fig:planar-pattern}
\end{figure}

Instead of considering only directed paths (with steps $a$ and $b$), as usual, in this paper, we allow all paths between two vertices that consist of steps $(a,b)$, $(a,-b)$, or $(-a,b)$. Thus, intuitively, the `distance' between two vertices $i,j$ is defined as the minimum number of these steps in such a path from $i$ to $j$, and the diameter is defined consistently. More rigorously,  
    for $x,y\in \mathbb Z_N$, we define $d(x,y):=d_{TQ}(x,y)$ as the minimum integer $\ell\geq 0$
    such that $y-x$ can be written modulo $N$ as
    \[
        y-x \equiv ma+nb,\qquad
        y-x \equiv -ma+nb,\qquad \text{or}\qquad
        y-x \equiv ma-nb,
    \]
    for some integers $m,n\geq 0$ with $m+n=\ell$.
    Equivalently, $d(x,y)$ is the minimum length of a word from $x$ to $y$
    using one of the three admissible directions determined by the pairs
    $(+a,+b)$, $(-a,+b)$, and $(+a,-b)$.
    Notice that this distance is translation invariant:
    \[
        d(x,y)=d(0,y-x).
    \]
    However, it is not necessarily symmetric, since the admissible directions do
    not include the pair $(-a,-b)$. Therefore, $d$ should be understood as a
    sector-constrained distance, rather than as a metric in the usual
    undirected sense. The diameter of $TQ(N,a,b)$ is then defined by
    \[
        k:=k(TQ(N,a,b))
        =
        \max_{x,y\in \mathbb Z_N} d(x,y)
        =
        \max_{z\in \mathbb Z_N} d(0,z).
    \]

\begin{figure}[t]
    \centering
    \includegraphics[width=6cm]{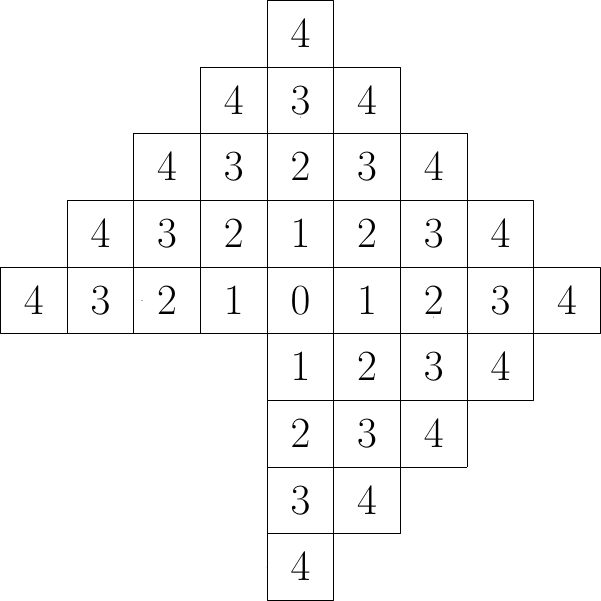}
    \caption{The distances from 0 to different vertices.}
    \label{fig:distances}
\end{figure}

These digraphs can be represented as congruent tiles that tessellate the plane periodically.
 More precisely, if each vertex is represented by a numbered modulo $N$ unit square, the vertices reached at distances $0,1,2,\ldots$ from any given vertex can be arranged in a planar pattern, as shown in Figure \ref{fig:planar-pattern} starting from vertex $0$.
% Thus, the vertices that are successively reached from vertex 0 can be arranged in a planar pattern, as shown in Figure \ref{fig:planar-pattern}.
From now on, we call these digraphs {\em three-quarters digraphs} (because only three of the four directions in the plane from a vertex 0 are admissible).
Then, Figure \ref{fig:distances} shows the distances from 0 to 4 of the different vertices and, in Figure \ref{fig:Tessellation}, there is the plane tessellation representing the three-quarter digraph $TQ(27,2,7)$.

Since there are $3\ell+1$ vertices at distance $\ell(>0)$ from vertex $0$,
given a diameter $k$, the number of vertices reached by exactly $i$ steps is $1,4,7,10,\ldots,3k+1$ (see Figure \ref{fig:distances}). Then, the maximum number of vertices  of a three-quarters digraph with diameter $k$ (ball size) would be
\begin{equation}
\label{eq:moore}
N(k)=\sum_{r=0}^k(1+3r)=
%k+1+3\sum_{r=0}^k r=
\frac{1}{2}\left(3k^2+5k+2\right)
\end{equation}
if all the numbers $ma+nb$, $-ma+nb$, and $ma-nb$ (with $m,n\ge 0$ and $m+n\le k$) were distinct modulo $N(k)$.
However, such a maximum cannot be attained for $k>1$. The reason is that, for $k>1$, the optimal tiles do not tessellate the plane.

\begin{figure}[t]
    \centering
    \includegraphics[width=6cm]{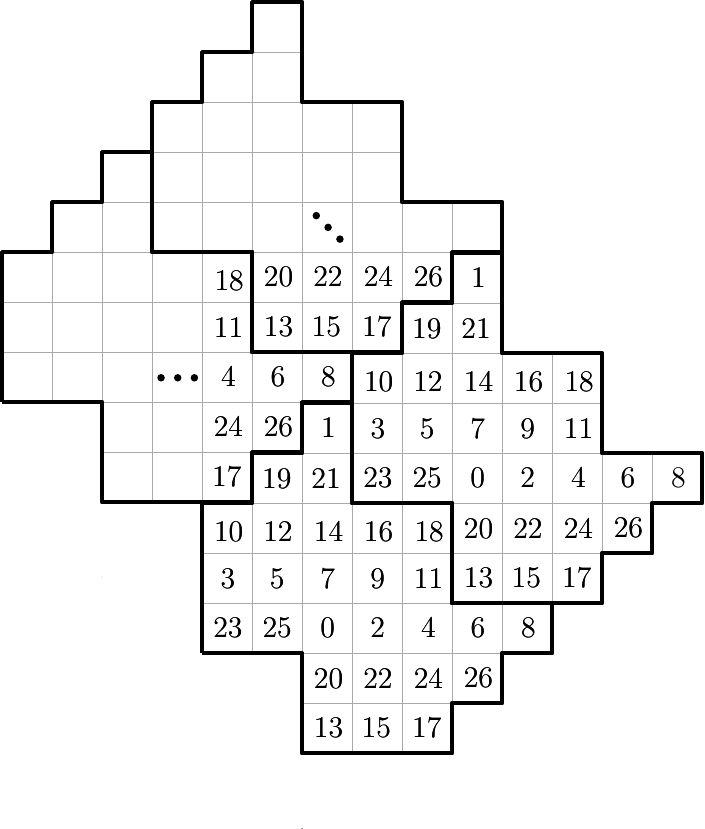}
    \vskip-0.5cm
    \caption{The plane tessellation representing the three-quarters digraph $TQ(27,2,7)$.}
    \label{fig:Tessellation}
\end{figure}

The method used in \cite{eaf93, mff85,
 yfma85} to obtain circulant digraphs with the maximum number of vertices consists of the following two steps: 
\begin{itemize}
\item
[{\bf $(1)$}] 
Using the planar pattern, find the `optimal tiles' (that is, containing the maximum number of vertices for a given diameter $k$), and check that they periodically tessellate the plane. %In our case, we have the following conjecture.
  
\item 
[{\bf $(2)$}]
From the two basic translation vectors of the periodic tiling (generating the lattice of the positions of the $0$ vertices), solve a linear system of equations to find the steps $a$ and $b$ satisfying {\bf P2.}
\end{itemize}

In other words, the lattice construction can be described as follows.  In the planar representation, moving one unit in the horizontal direction corresponds to adding $a$, while moving one unit in the vertical direction corresponds to adding $b$. Hence, a translation vector $(u,v)$ in the plane corresponds to
    the residue
    \[
        ua+vb \pmod N.
    \]
    If two copies of the tile differ by a lattice translation, then they must
    represent the same vertex labels modulo $N$. Therefore, for each translation
    vector $(u,v)$ of the lattice one must have
    \[
        ua+vb\equiv 0 \pmod N.
    \]
    If the lattice is generated by two independent translation vectors
    \[
        (u_1,v_1),\qquad (u_2,v_2),
    \]
    then the generators $a,b$ must satisfy
    \[
        u_1a+v_1b\equiv 0 \pmod N,
        \qquad
        u_2a+v_2b\equiv 0 \pmod N.
    \]
    Equivalently, there exist integers $\alpha,\beta$ such that
    \[
        \begin{pmatrix}
        u_1 & v_1\\
        u_2 & v_2
        \end{pmatrix}
        \begin{pmatrix}
        a\\ b
        \end{pmatrix}
        =
        N
        \begin{pmatrix}
        \alpha\\ \beta
        \end{pmatrix}.
    \]
    The parameters $\alpha$ and $\beta$ record how many full periods modulo $N$
    are produced by the two fundamental lattice translations. Once the tile is
    fixed, the determinant
    \[
        N=
        \det
        \begin{pmatrix}
        u_1 & v_1\\
        u_2 & v_2
        \end{pmatrix},
    \]
    which is the area of the fundamental parallelogram spanned by the two lattice vectors,
    gives the number of distinct residue classes (labels of vertices) in one fundamental tile, see Esqué, Aguiló, and Fiol \cite{eaf93}.
    This is why the order $N$ of the digraph is directly determined by the lattice.
    Solving the above linear system then yields admissible values of steps $a$ and $b$ provided that $\alpha$ and $\beta$ are chosen in such a way that $\gcd(N,a,b)=1$.
    This last condition {\bf P2} assures that all the residues modulo $N$ are obtained.
%     \blue{Thus, the fundamental result is the following lemma.}
% \blue{
% \begin{lemma} (Fundamental-domain criterion)
% Let $L\subset \Z^2$ be a lattice of determinant $N$, and let $W\subset \Z^2$ satisfy
% \begin{itemize}
% \item {\bf 1.} $|W|=N$,
% \item {\bf 2.}
% No two points of $W$ differ by a vector of $L$.
% \end{itemize}
% Then, the map
% $$
% (x,y)\mapsto xa+yb \pmod{N}
% $$
% where $\gcd(a,b,N)=1$,
% induces a bijection between $W$ and $\mathbb{Z_N}$.
% \end{lemma}
% }
Moreover, the value of the diameter $k$ is imposed by the form of the tile (for instance, $k=4$ if and only if the tile is a sub-tile of the one in Figure \ref{fig:distances} containing some squares labeled with 4). 
Given a diameter $k>1$, a good choice for the tiles that periodically tessellate the plane is the $W$-shaped tiles shaded in  Figure \ref{fig:optimal-tiles} for $k=4$ and $k=5$. (The case $k=1$ is trivial, giving the cross with arms of one unit).
In such a figure, the vertices at the farthest distance from 0 are indicated by black dots.
\begin{figure}[t]
    \centering
\includegraphics[width=10cm]{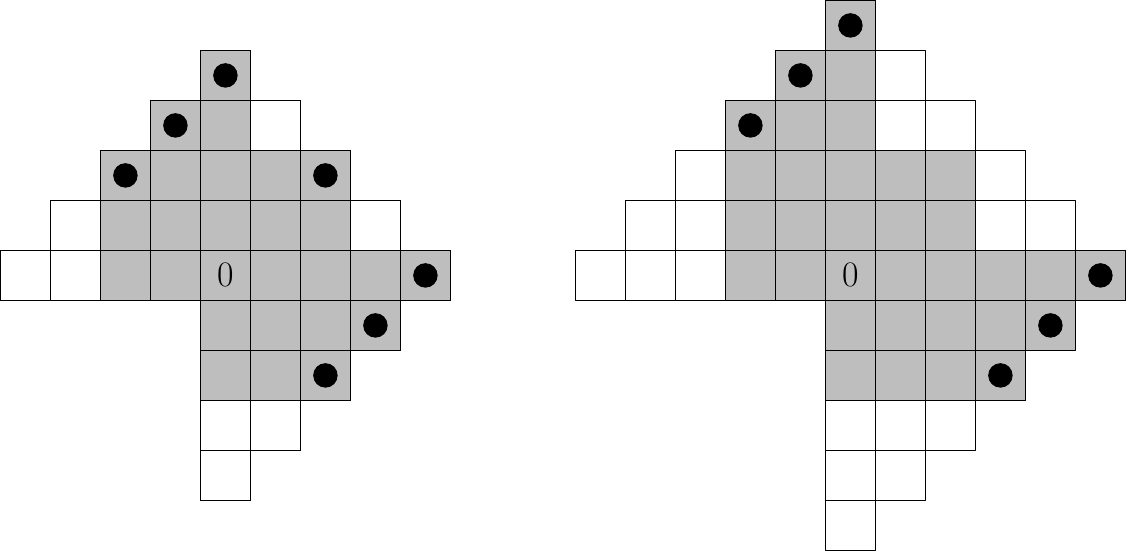}
    \caption{In gray, there are the optimal $W$-tiles with even ($k=4$) and odd diameter ($k=5$). The farthest vertices from $0$ are indicated by black dots.}
    \label{fig:optimal-tiles}
\end{figure}
In our study of such tiles, we have to distinguish between even and odd diameters $ k$.

\begin{proposition}
\label{propo:even}
There exists a three-quarter digraph $TQ(N,a,b)$ for any even diameter $k=2n$, $N=k^2+\frac{5}{2}k+1$ vertices, and steps $a=n$ and $b=3n+1$.
\end{proposition}
\begin{proof}
For even $k$, that is, $k=2n$, the two vectors generating the lattice are $(n+1,n+1)$ and $(-2n,2n+1)$. See Figure \ref{fig:lattices} (left). Then, the equations for the distribution of the zeros (of the lattice) are
$$
(n+1)a+(n+1)b\equiv 0 \ (\textrm{mod }N),\quad\mbox{and}\quad -2na+(2n+1)b\equiv 0 \ (\textrm{mod }N).
$$
Then, the pair of steps $a$ and $b$ satisfies the system
$$
\left(\begin{array}{cc}
n+1 & n+1 \\
-2n & 2n+1
\end{array} \right)
\left(
\begin{array}{c}
a \\
b
\end{array}
\right)=N\left(\begin{array}{c}
\alpha \\
\beta
\end{array}
\right)
$$
for some integers $\alpha,\beta$ and  $N=\det\left(\begin{array}{cc}
n+1 & n+1 \\
-2n & 2n+1
\end{array} \right)=4n^2+5n+1=k^2+\frac{5}{2}k+1$. Solving the system, we have
$$
\left(
\begin{array}{c}
a \\
b
\end{array}
\right)=
\left(\begin{array}{cc}
2n+1 & -(n+1) \\
2n & n+1
\end{array} \right)\left(\begin{array}{c}
\alpha \\
\beta
\end{array}
\right).
$$
Taking, for instance, $\alpha=\beta=1$, we get the steps $a=n$ and $b=3n+1$ satisfying Property {\bf P2}. In the case of diameter $k=2$, or $n=1$, Figure \ref{fig:circulant} shows the three-quarters digraph $TQ(10,1,4)$.
\end{proof}

\begin{figure}[t]
    \centering
    \includegraphics[width=12cm]{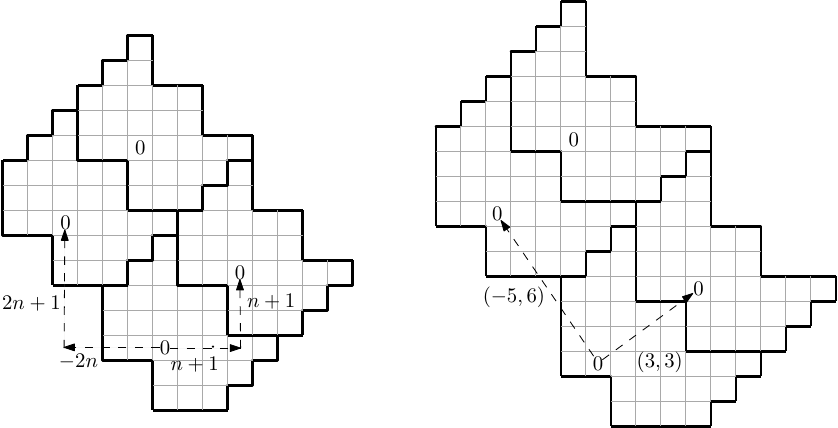}
    \caption{Plane tessellation with $W$-tiles, and vectors generating the lattices. Left: $k=2n=4$; Right: $k=2n+1=5$.}
    \label{fig:lattices}
\end{figure}

\begin{proposition}
\label{propo:odd}
There exists a three-quarter digraph $TQ(N,a,b)$ for any odd diameter $k=2n+1>1$, $N=k^2+\frac{3}{2}k+\frac{1}{2}$ vertices, and steps $a=n+1$ and $b=3n+2$.
\end{proposition}
\begin{proof}
For odd $k$, that is, $k=2n+1$, the two vectors generating the lattice are $(n+1,n+1)$ and $(-(2n+1),2n+2)$. See Figure \ref{fig:lattices} (right). Then, the pair of steps satisfies the system
$$
\left(\begin{array}{cc}
n+1 & n+1 \\
-(2n+1) & 2n+2
\end{array} \right)
\left(
\begin{array}{c}
a \\
b
\end{array}
\right)=N\left(\begin{array}{c}
\alpha \\
\beta
\end{array}
\right),
$$
where $N=\det\left(\begin{array}{cc}
n+1 & n+1 \\
-(2n+1) & 2n+2
\end{array} \right)=4n^2+7n+3=k^2+\frac{3}{2}k+\frac{1}{2}$. Solving the system, we have
$$
\left(
\begin{array}{c}
a \\
b
\end{array}
\right)=
\left(\begin{array}{cc}
2n+2 & -(n+1) \\
2n+1 & n+1
\end{array} \right)\left(\begin{array}{c}
\alpha \\
\beta
\end{array}
\right).
$$
Taking, for instance, $\alpha=\beta=1$, we get the steps $a=n+1$ and $b=3n+2$.
\end{proof}

\begin{figure}[t]
    \centering
\includegraphics[width=10cm]{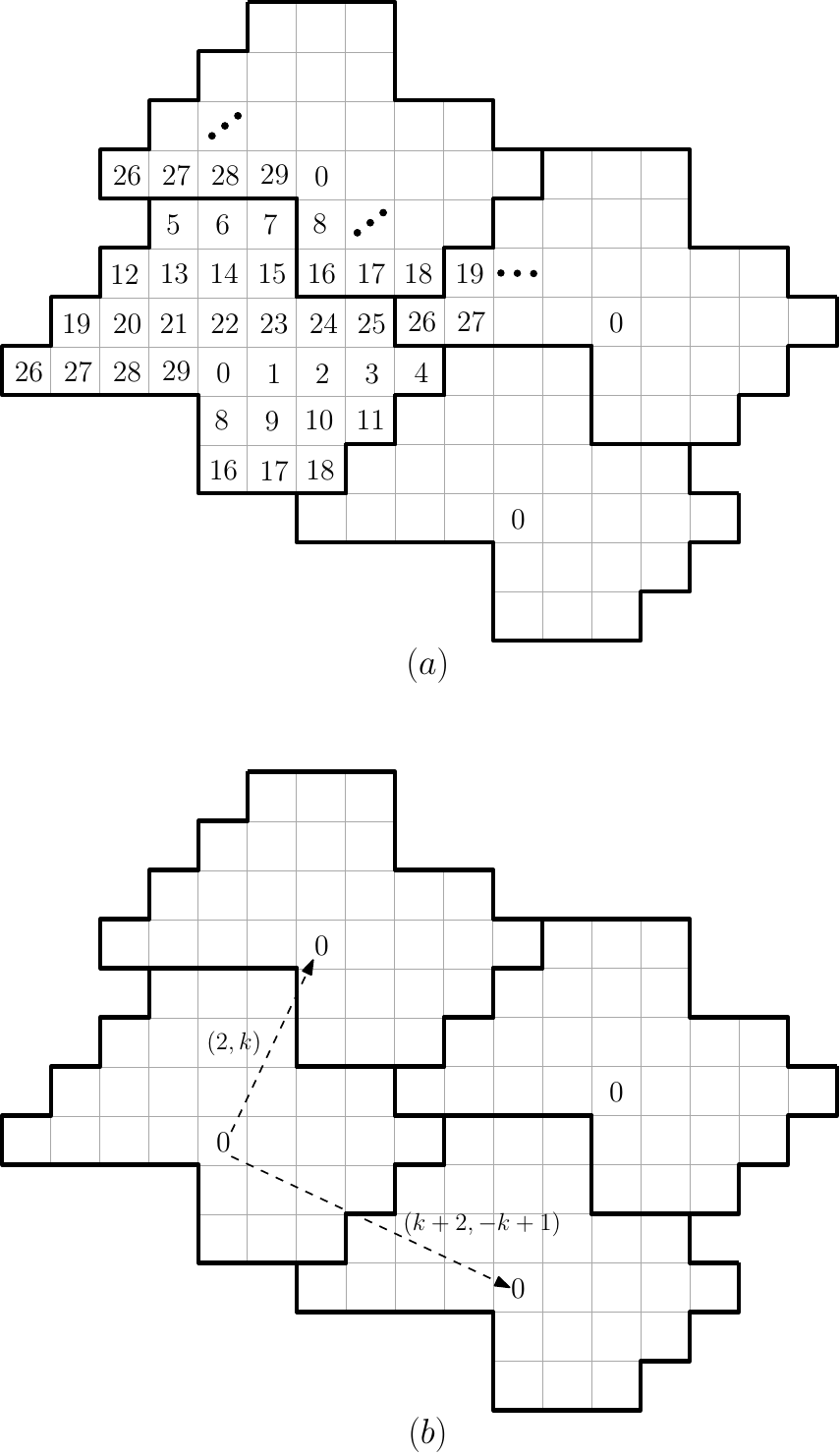}
    \caption{$(a)$ Plane tessellation for the three-quarters digraph with diameter $k=4$, order $N=30$, and  steps $a=1$ and $b=22$.
    $(b)$ Conjectured vectors generating the lattices for a general tile.}
    \label{fig:teselk=4}
\end{figure}

For the particular case $k=1$, we have the three-quarter digraph $TQ(5,1,2)$.

In Table \ref{tab:tabla1-tq}, we show the minimum diameter $k$ and steps $a,b$ for each number of vertices $N\le 208$ of three-quarters digraphs. The cases in which we get the maximum number of vertices for a given diameter are in boldface. Notice that such a maximum is greater than the values obtained with the $W$-shape tiles. For instance, for diameter $k=4$, Proposition \ref{propo:even} gives $N=27$ (see Figure \ref{fig:optimal-tiles}), whereas, according to Table \ref{tab:tabla1-tq}, we have a maximum of $N=30$. The tiles that achieve the maximum values are more involved, as shown in Figure \ref{fig:teselk=4} for $N=30$ and steps $a=1$ and $b=22=-8$.
% Looking at Table \ref{tab:tabla2-tq} for even diameter $k$, the maximum value is $N=\frac{9}{8}k^2+\frac{11}{4}k+1$, showing that the value $N=k^2+\frac{5}{2}k+1$ of Proposition \ref{propo:even} yields a good approximation.

As it was given in \eqref{eq:moore}, for diameter $k$, the number of vertices contained in the ideal ball of radius $k$ is $\frac{3}{2}k^2+\frac{5}{2}k+1$.
    This gives the natural Moore-type upper bound for the three-quarter distance
    model, provided that all represented residues are distinct modulo $N$.
    Nevertheless, for $k>1$, this ideal ball does not tessellate the plane. Hence, the upper bound is not achievable
 with a three-quarters circulant digraph.
    The constructions given in Propositions 2.1 and 2.2 should therefore be
    interpreted as explicit infinite families of large order, rather than as
    optimal constructions in general. More precisely, they give
    \[
        N(k)=k^2+\frac52k+1
        \quad\text{for even }k,
    \]
    and
    \[
        N(k)=k^2+\frac32k+\frac12
        \quad\text{for odd }k.
    \]
In contrast, according to Table \ref{tab:tabla3-tq}, the maximum order for even diameter $k\le 30$ is $N(k)=\frac{9}{8}k^2 + \frac{22}{8}k + 1$, achieved with steps $a=1$ and
$b(k)=7+23\left(\frac{k-2}{4}\right)+18\left(\frac{k-2}{4}\right)^2$ for $k=2r$ with odd $r$, and 
$b(k)=22+41\left(\frac{k-4}{4}\right)+18\left(\frac{k-4}{4}\right)^2$ for $k=2r$, with even $r$.
Experimental results suggest that these values are valid for any even diameter. The difficulty is that the associated tiles exhibit similar but apparently non-generalizable shapes. Examples include the cases $k=6,12,18$ in Figure \ref{fig:teselk=6-12-18}.
 % \blue{As commented above,} for even diameter, computational evidence suggests that larger examples may
 %    exist, with
 %    \[
 %        N(k)=\frac98k^2+\frac{22}{8}k+1.
 %    \]
    Thus, one may add the following conjecture.
    \begin{conjecture}
       For every even diameter $k$, there exists a three-quarter circulant digraph
    $TQ(N,a,b)$ with
    \[
        N(k)=\frac98k^2+\frac{22}{8}k+1,
    \]
    and steps indicated as above. More precisely: if $k=4s$ we have $N(k)=18s^2+11s+1$, $a=1$, and $b(k)=18s^2+5s-1$; and if $k=4s+2$ we have $N(k)=18s^2+29s+11$, $a=1$, and $b(k)=18s^2+23s+7$. Moreover, such numbers of vertices $N(k)$ are optimal among all three-quarter circulant digraphs of such diameters $k$. 
    \end{conjecture}

\begin{figure}[t]
    \centering
    \includegraphics[width=10cm]{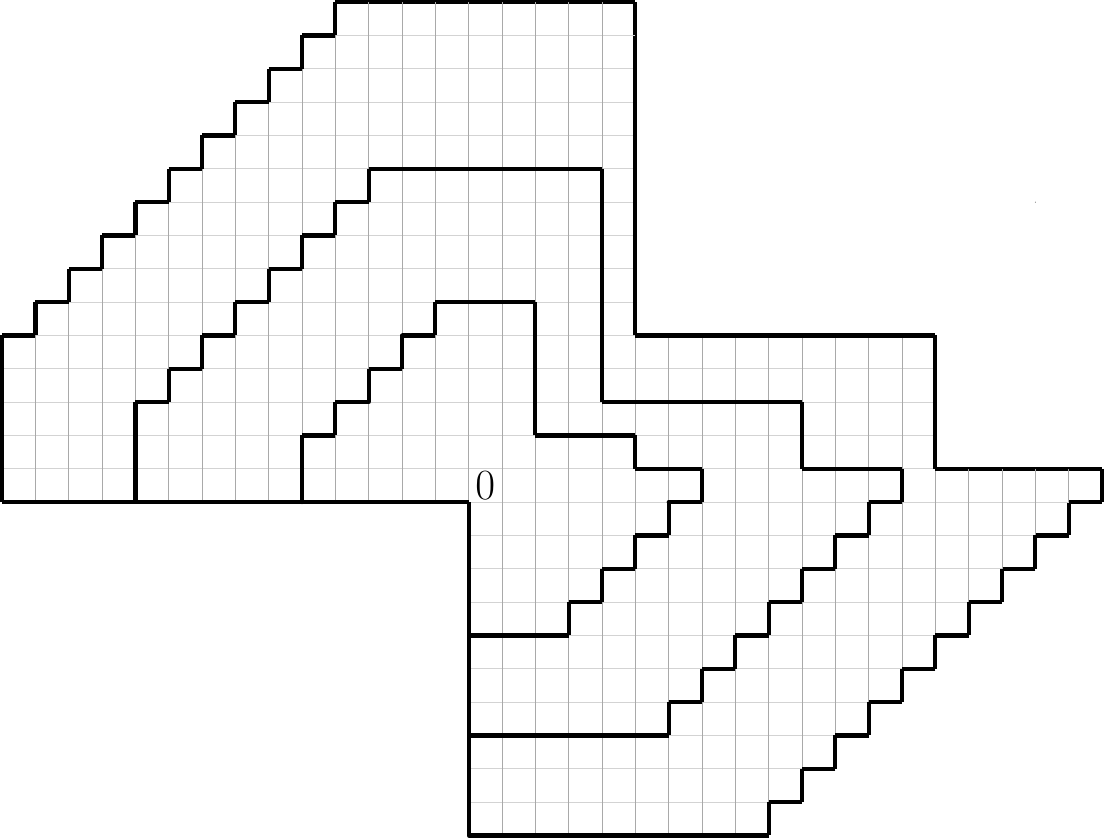}
    \caption{Similar plane tiles for the three-quarters digraphs with diameters $k=6$, $12$, and $18$, and maximum number of vertices.}
    \label{fig:teselk=6-12-18}
\end{figure}

From the case of $k=4$, we also have the following conjecture 
for general diameter, slightly improving that of Propositions \ref{propo:even} and \ref{propo:odd}.

\begin{conjecture}
\label{propo:even(b)}
Given an integer $k\ge 1$, the lattice generated by the vectors $(u_1,v_1)=(k+2,-k+1)$ and $(u_2,v_2)=(2,k)$ corresponds to a tile of a three-quarters digraph with diameter $k$.
Then, there exists a three-quarter digraph $TQ(N,a,b)$ with diameter $k$, $N=k^2+4k-2$ vertices, and steps $a=1$ and $b=-k-4$.
\end{conjecture}

Looking at the case for $k=4$ in the tables and also Figures \ref{fig:teselk=4} $(a)$ and $(b)$, the equations for the distribution of the zeros could be
$$
(k+2)a+(-k+1)b\equiv 0 \ (\textrm{mod }N),\quad\mbox{and}\quad 2a+kb\equiv 0 \ (\textrm{mod }N).
$$
Then, the pair of steps $a$ and $b$ satisfies the system
$$
\left(\begin{array}{cc}
k+2 & -k+1 \\
2 & k
\end{array} \right)
\left(
\begin{array}{c}
a \\
b
\end{array}
\right)=N\left(\begin{array}{c}
\alpha \\
\beta
\end{array}
\right)
$$
for some integers $\alpha,\beta$ and  $N=\det\left(\begin{array}{cc}
k+2 & -k+1 \\
2 & k
\end{array} \right)=k^2+4k-2$. Solving the system, we have
$$
\left(
\begin{array}{c}
a \\
b
\end{array}
\right)=
\left(\begin{array}{cc}
k & k-1 \\
-2 & k+2
\end{array} \right)\left(\begin{array}{c}
\alpha \\
\beta
\end{array}
\right).
$$
Taking, for instance, $\alpha=1$ and $\beta=-1$, we get the steps $a=1$ and $b=-k-4$ satisfying property {\bf P2}.

%%%%%%%%%%%%%%%%%%%%%%%%%%%%%%%%%%%%%%%%%%%%%%%%%%%%%
\section*{Acknowledgments}

The authors thank emergency medical doctor Susana Sim\'o for her help with the name of these new digraphs. 

%%%%%%%%%%%%%%%%%%%%%%%%%%%%%%%%%%%%%%%%%%%%%%%%%%%%%
\section*{Statements and Declarations}

The authors have no competing interests.\\
All the data generated in this paper is included here. 

%%%%%%%%%%%%%%%%%%%%%%%%%%%%%%%%%%%%%%%%%%%%%%%%%%%%%

%\section*{Appendix }
\appendix
  \begin{table}[t]
  \caption{The minimum diameter $k$ for each number of vertices $N\le 208$ with steps $a$ and $b$ in the three-quarters digraphs $TQ(N,a,b)$.}
		\label{tab:tabla1-tq}
  \small
		\begin{center}
\setlength{\tabcolsep}{3pt}
\begin{tabular}{
||c|c|c|c||
|c|c|c|c||
|c|c|c|c||
|c|c|c|c||
|c|c|c|c||
|c|c|c|c||}
\hline
$N$ & $a$ & $b$ & $k$ & $N$ & $a$ & $b$ & $k$ &
$N$ & $a$ & $b$ & $k$ &
$N$ & $a$ & $b$ & $k$ &
$N$ & $a$ & $b$ & $k$  \\
\hline
4 &   1 &   2 &   1 &	45 &   1 &  20 &   6 &	86 &   1 &  20 &   8 &	127 &   1 &  24 &  10 &	168 &   1 &  15 &  12 \\
{\bf 5} &   {\bf 1} &  {\bf 2} &  {\bf 1} &	46 &   1 &  14 &   6 &	87 &   1 &  12 &   8 &	128 &   1 &  30 &  10 &	169 &   1 &  27 &  12 \\
6 &   1 &   2 &   2 &	47 &   1 &   9 &   6 &	88 &   1 &  14 &   8 &	129 &   1 &  14 &  10 &	170 &   1 &  14 &  12 \\
7 &   1 &   2 &   2 &	48 &   1 &   9 &   6 &	89 &   1 &  41 &   8 &	130 &   1 &  60 &  10 &	171 &   1 &  54 &  12 \\
8 &   1 &   3 &   2 &	49 &   1 &  15 &   6 &	90 &   1 &  66 &   8 &	131 &   1 &  18 &  10 &	172 &   1 &  67 &  12 \\
9 &   1 &   4 &   2 &	50 &   1 &   8 &   6 &	91 &   1 &  28 &   8 &	132 &   1 &  16 &  10 &	173 &   1 &  17 &  12 \\
10 &   1 &   4 &   2 &	51 &   1 &  30 &   6 &	92 &   1 &  76 &   8 &	133 &   1 &  62 &  10 &	174 &   1 &  17 &  12 \\
{\bf 11} &   {\bf 1} &   {\bf 7} &  {\bf 2} &	52 &   1 &  12 &   6 &	93 &   1 &  13 &   9 &	134 &   1 &  12 &  11 &	175 &   1 &  19 &  12 \\
12 &   1 &   4 &   3 &	53 &   1 &  10 &   6 &	94 &   1 &  82 &   8 &	135 &   1 & 115 &  10 &	176 &   1 &  19 &  12 \\
13 &   1 &   3 &   3 &	54 &   1 &  24 &   6 &	{\bf 95} &   {\bf 1} &  {\bf 61} &   {\bf 8} &	136 &   1 &  32 &  10 &	177 &   1 &  69 &  12 \\
14 &   1 &   4 &   3 &	55 &   1 &  35 &   6 &	96 &   1 &  18 &   9 &	137 &   1 &  15 &  11 &	178 &   1 &  42 &  12 \\
15 &   1 &   6 &   3 &	56 &   1 &   9 &   7 &	97 &   1 &  30 &   9 &	138 &   1 & 124 &  10 &	179 &   1 &  16 &  12 \\
16 &   1 &  10 &   3 &	57 &   1 &  45 &   6 &	98 &   1 &  12 &   9 &	139 &   1 &  15 &  11 &	180 &   1 &  84 &  12 \\
17 &   1 &   5 &   3 &	{\bf 58} &  {\bf 1} &  {\bf 48} &  {\bf 6} &	99 &   1 &  12 &   9 &	140 &   1 & 122 &  10 &	181 &   1 &  82 &  12 \\
18 &   2 &  15 &   3 &	59 &   1 &  11 &   7 &	100 &   1 &  16 &   9 &	{\bf 141} &  {\bf 1} &  {\bf 44} & {\bf 10} &	182 &   1 &  22 &  12 \\
{\bf 19} &   {\bf 1} &   {\bf 8} &   {\bf 3} &	60 &   1 &  14 &   7 &	101 &   1 &  11 &   9 &	142 &   1 &  14 &  11 &	183 &   1 &  13 &  13 \\
20 &   1 &   6 &   4 &	61 &   1 &  14 &   7 &	102 &   1 &  14 &   9 &	143 &   1 &  14 &  11 &	184 &   1 &  18 &  12 \\
21 &   1 &   4 &   4 &	62 &   1 &  10 &   7 &	103 &   1 &  24 &   9 &	144 &   1 &  34 &  11 &	185 &   1 &  20 &  12 \\
22 &   1 &   5 &   4 &	63 &   1 &  10 &   7 &	104 &   1 &  32 &   9 &	145 &   1 &  13 &  11 &	186 &   1 & 137 &  12 \\
23 &   1 &   7 &   4 &	64 &   1 &  12 &   7 &	105 &   1 &  77 &   9 &	146 &   1 &  20 &  11 &	187 &   1 &  59 &  12 \\
24 &   1 &   7 &   4 &	65 &   1 &   9 &   7 &	106 &   1 &  20 &   9 &	147 &   1 &  16 &  11 &	188 &   1 &  30 &  13 \\
25 &   1 &  11 &   4 &	66 &   1 &  20 &   7 &	107 &   1 &  13 &   9 &	148 &   1 &  16 &  11 &	189 &   1 &  36 &  12 \\
26 &   1 &   6 &   4 &	67 &   1 &  26 &   7 &	108 &   1 &  42 &   9 &	149 &   1 &  18 &  11 &	190 &   1 & 140 &  12 \\
27 &   1 &   8 &   4 &	68 &   1 &  40 &   7 &	109 &   1 &  15 &   9 &	150 &   1 &  42 &  11 &	191 &   1 &  17 &  13 \\
28 &   1 &  12 &   4 &	69 &   1 &  11 &   7 &	110 &   1 &  96 &   9 &	151 &   1 &  24 &  11 &	192 &   1 &  42 &  13 \\
29 &   1 &  23 &   4 &	70 &   1 &  58 &   7 &	111 &   1 &  10 &  10 &	152 &   1 &  84 &  11 &	193 &   1 & 114 &  12 \\
{\bf 30} &   {\bf 1} &  {\bf 22} &   {\bf 4} &	71 &   1 &  52 &   7 &	112 &   1 &  72 &   9 &	153 &   1 &  15 &  11 &	194 &   1 &  16 &  13 \\
31 &   1 &   5 &   5 &	72 &   1 &  10 &   8 &	113 &   1 &  18 &  10 &	154 &   1 &  48 &  11 &	195 &   1 & 177 &  12 \\
32 &   1 &   6 &   5 &	73 &   1 &   8 &   8 &	114 &   1 &  14 &  10 &	155 &   1 & 128 &  11 &	{\bf 196} &   {\bf 1} & {\bf 176} & {\bf 12} \\
33 &   1 &  10 &   5 &	74 &   1 &   9 &   8 &	115 &   1 &  27 &   9 &	156 &   1 & 136 &  11 &	197 &   1 &  15 &  13 \\
34 &   1 &   8 &   5 &	{\bf 75} &  {\bf 1} &  {\bf 23} &  {\bf 7} &	116 &   1 &  14 &  10 &	157 &   1 &  17 &  11 &	198 &   1 &  24 &  13 \\
35 &   1 &   8 &   5 &	76 &   1 &  12 &   8 &	{\bf 117} &   {\bf 1} & {\bf 102} &   {\bf 9} &	158 &   1 & 140 &  11 &	199 &   1 &  24 &  13 \\
36 &   1 &  14 &   5 &	77 &   1 &  18 &   8 &	118 &   1 &  33 &  10 &	159 &   1 & 141 &  11 &	200 &   1 &  32 &  13 \\
37 &   1 &   7 &   5 &	78 &   1 &  18 &   8 &	119 &   1 &  13 &  10 &	160 &   1 &  50 &  11 &	201 &   1 &  18 &  13 \\
38 &   1 &  24 &   5 &	79 &   1 &  11 &   8 &	120 &   1 &  13 &  10 &	161 &   1 &  22 &  12 &	202 &   1 &  18 &  13 \\
39 &   1 &   9 &   5 &	80 &   1 &  11 &   8 &	121 &   1 &  23 &  10 &	162 &   2 & 117 &  11 &	203 &   1 &  22 &  13 \\
40 &   1 &  12 &   5 &	81 &   1 &  13 &   8 &	122 &   1 &  12 &  10 &	163 &   1 &  31 &  11 &	204 &   1 &  20 &  13 \\
41 &   1 &  26 &   5 &	82 &   1 &  10 &   8 &	123 &   1 &  15 &  10 &	164 &   1 &  16 &  12 &	205 &   1 &  20 &  13 \\
42 &   1 &   8 &   6 &	83 &   1 &  47 &   8 &	124 &   1 &  15 &  10 &	165 &   1 &  18 &  12 &	206 &   1 &  49 &  13 \\
{\bf 43} &   {\bf 1} &  {\bf 19} &  {\bf 5} &	84 &   1 &  38 &   8 &	125 &   1 &  35 &  10 &	166 &   1 &  18 &  12 &	207 &   1 &  17 &  13 \\
44 &   1 &   7 &   6 &	85 &   1 &  16 &   8 &	126 &   1 &  20 &  10 &	{\bf 167} &   {\bf 1} & {\bf 108} &  {\bf 11} &	208 &   1 &  88 &  13 \\ 
\hline
\end{tabular}
		\end{center}
		\end{table}

%  \begin{table}[t]
%   \caption{The maximum number of vertices $N(k)$ with steps $a$ and $b$ for diameter $k$ in the three-quarters digraph $TQ(N,a,b)$.}
% \label{tab2:tabla-tq}
% \small
% 		\begin{center}
% \setlength{\tabcolsep}{3pt}
% \begin{tabular}{
% ||c|c|c|c||}
% \hline
% $N$ & $a$ & $b$ & $k$  \\
% \hline
% 5  &    1  &   2  &   1  \\
%  11  &    1  &    7  &   2  \\
%  19  &    1  &    8  &    3  \\
%  30  &    1  &   22  &    4  \\
%  43  &    1  &   19  &   5  \\
%  58  &   1  &   48  &   6  \\
%  75  &   1  &   23  &   7  \\
%  95  &    1  &   61  &    8\\
%  117  &    1  &  102  &    9  \\
%  141  &   1  &   44  &  10\\
%  167  &    1  &  108  &   11  \\
%  196  &    1  &  176  &  12  \\
% \hline
% \end{tabular}
% 		\end{center}
% 		\end{table}

 \begin{table}[t]
  \caption{ Values of $N, a, b,$ and $k$ with different steps $a$ and $b$ for $N=\frac{9k^2}{8} + \frac{22k}{8}+1$ and even $k$ in the three-quarters digraphs $TQ(N,a,b)$.}
		\label{tab:tabla2-tq}
  \small
		\begin{center}
\setlength{\tabcolsep}{3pt}
\begin{tabular}{
||c|c|c|c||
|c|c|c|c||
|c|c|c|c||
|c|c|c|c||
|c|c|c|c||
|c|c|c|c||}
\hline
$N$ & $a$ & $b$ & $k$ & $N$ & $a$ & $b$ & $k$ &
$N$ & $a$ & $b$ & $k$ &
$N$ & $a$ & $b$ & $k$ &
$N$ & $a$ & $b$ & $k$  \\
\hline
11 & 1 & 7 & 2 &	58 & 35 & 56 & 6 &	95 & 27 & 32 & 8 &	141 & 4 & 77 & 10 &	141 & 41 & 49 & 10 \\
11 & 1 & 8 & 2 &	58 & 36 & 37 & 6 &	95 & 28 & 83 & 8 &	141 & 5 & 61 & 10 &	141 & 41 & 112 & 10 \\
11 & 2 & 3 & 2 &	58 & 41 & 54 & 6 &	95 & 28 & 93 & 8 &	141 & 5 & 79 & 10 &	141 & 43 & 59 & 10 \\
11 & 2 & 5 & 2 &	58 & 47 & 52 & 6 &	95 & 29 & 59 & 8 &	141 & 7 & 26 & 10 &	141 & 44 & 103 & 10 \\
11 & 3 & 10 & 2 &	58 & 50 & 53 & 6 &	95 & 29 & 69 & 8 &	141 & 7 & 29 & 10 &	141 & 46 & 50 & 10 \\
11 & 4 & 6 & 2 &	95 & 1 & 61 & 8 &	95 & 31 & 41 & 8 &	141 & 8 & 13 & 10 &	141 & 46 & 110 & 10 \\
11 & 4 & 10 & 2 &	95 & 1 & 81 & 8 &	95 & 31 & 86 & 8 &	141 & 8 & 70 & 10 &	141 & 49 & 62 & 10 \\
11 & 5 & 7 & 2 &	95 & 2 & 27 & 8 &	95 & 32 & 52 & 8 &	141 & 10 & 17 & 10 &	141 & 50 & 85 & 10 \\
11 & 6 & 9 & 2 &	95 & 2 & 67 & 8 &	95 & 34 & 79 & 8 &	141 & 10 & 122 & 10 &	141 & 53 & 76 & 10 \\
11 & 8 & 9 & 2 &	95 & 3 & 53 & 8 &	95 & 34 & 94 & 8 &	141 & 11 & 61 & 10 &	141 & 53 & 139 & 10 \\
30 & 1 & 22 & 4 &	95 & 3 & 88 & 8 &	95 & 36 & 66 & 8 &	141 & 11 & 106 & 10 &	141 & 55 & 107 & 10 \\
30 & 2 & 11 & 4 &	95 & 4 & 39 & 8 &	95 & 37 & 52 & 8 &	141 & 13 & 74 & 10 &	141 & 56 & 67 & 10 \\
30 & 4 & 7 & 4 &	95 & 4 & 54 & 8 &	95 & 37 & 72 & 8 &	141 & 14 & 52 & 10 &	141 & 56 & 91 & 10 \\
30 & 8 & 29 & 4 &	95 & 6 & 11 & 8 &	95 & 41 & 91 & 8 &	141 & 14 & 58 & 10 &	141 & 58 & 59 & 10 \\
30 & 13 & 16 & 4 &	95 & 6 & 81 & 8 &	95 & 42 & 77 & 8 &	141 & 16 & 26 & 10 &	141 & 62 & 136 & 10 \\
30 & 14 & 17 & 4 &	95 & 7 & 47 & 8 &	95 & 42 & 92 & 8 &	141 & 16 & 140 & 10 &	141 & 64 & 104 & 10 \\
30 & 19 & 28 & 4 &	95 & 7 & 92 & 8 &	95 & 43 & 58 & 8 &	141 & 17 & 43 & 10 &	141 & 64 & 137 & 10 \\
30 & 23 & 26 & 4 &	95 & 8 & 13 & 8 &	95 & 43 & 63 & 8 &	141 & 19 & 119 & 10 &	141 & 65 & 88 & 10 \\
58 & 1 & 48 & 6 &	95 & 8 & 78 & 8 &	95 & 44 & 49 & 8 &	141 & 19 & 131 & 10 &	141 & 67 & 128 & 10 \\
58 & 2 & 23 & 6 &	95 & 9 & 64 & 8 &	95 & 46 & 51 & 8 &	141 & 20 & 34 & 10 &	141 & 70 & 119 & 10 \\
58 & 3 & 28 & 6 &	95 & 9 & 74 & 8 &	95 & 48 & 78 & 8 &	141 & 20 & 103 & 10 &	141 & 71 & 133 & 10 \\
58 & 4 & 17 & 6 &	95 & 11 & 36 & 8 &	95 & 48 & 88 & 8 &	141 & 22 & 71 & 10 &	141 & 73 & 101 & 10 \\
58 & 5 & 8 & 6 &	95 & 12 & 22 & 8 &	95 & 49 & 74 & 8 &	141 & 22 & 122 & 10 &	141 & 73 & 110 & 10 \\
58 & 6 & 11 & 6 &	95 & 12 & 67 & 8 &	95 & 51 & 71 & 8 &	141 & 23 & 25 & 10 &	141 & 74 & 85 & 10 \\
58 & 7 & 46 & 6 &	95 & 13 & 33 & 8 &	95 & 54 & 64 & 8 &	141 & 23 & 55 & 10 &	141 & 76 & 101 & 10 \\
58 & 9 & 26 & 6 &	95 & 14 & 89 & 8 &	95 & 56 & 71 & 8 &	141 & 25 & 113 & 10 &	141 & 79 & 92 & 10 \\
58 & 10 & 57 & 6 &	95 & 14 & 94 & 8 &	95 & 56 & 91 & 8 &	141 & 28 & 104 & 10 &	141 & 80 & 130 & 10 \\
58 & 12 & 51 & 6 &	95 & 16 & 26 & 8 &	95 & 59 & 84 & 8 &	141 & 28 & 116 & 10 &	141 & 80 & 136 & 10 \\
58 & 13 & 44 & 6 &	95 & 16 & 61 & 8 &	95 & 62 & 77 & 8 &	141 & 29 & 100 & 10 &	141 & 82 & 83 & 10 \\
58 & 14 & 45 & 6 &	95 & 17 & 47 & 8 &	95 & 62 & 82 & 8 &	141 & 31 & 68 & 10 &	141 & 82 & 98 & 10 \\
58 & 15 & 24 & 6 &	95 & 17 & 87 & 8 &	95 & 63 & 68 & 8 &	141 & 31 & 95 & 10 &	141 & 83 & 127 & 10 \\
58 & 16 & 39 & 6 &	95 & 18 & 33 & 8 &	95 & 68 & 93 & 8 &	141 & 32 & 52 & 10 &	141 & 86 & 118 & 10 \\
58 & 18 & 33 & 6 &	95 & 18 & 53 & 8 &	95 & 69 & 79 & 8 &	141 & 32 & 139 & 10 &	141 & 89 & 109 & 10 \\
58 & 19 & 42 & 6 &	95 & 21 & 46 & 8 &	95 & 73 & 83 & 8 &	141 & 34 & 86 & 10 &	141 & 89 & 127 & 10 \\
58 & 20 & 27 & 6 &	95 & 21 & 86 & 8 &	95 & 82 & 87 & 8 &	141 & 35 & 130 & 10 &	141 & 91 & 95 & 10 \\
58 & 21 & 22 & 6 &	95 & 22 & 72 & 8 &	95 & 84 & 89 & 8 &	141 & 37 & 77 & 10 &	141 & 92 & 100 & 10 \\
58 & 25 & 40 & 6 &	95 & 23 & 58 & 8 &	141 & 1 & 44 & 10 &	141 & 37 & 113 & 10 &	141 & 97 & 140 & 10 \\
58 & 30 & 55 & 6 &	95 & 23 & 73 & 8 &	141 & 1 & 125 & 10 &	141 & 38 & 97 & 10 &	141 & 98 & 124 & 10 \\
58 & 31 & 38 & 6 &	95 & 24 & 39 & 8 &	141 & 2 & 88 & 10 &	141 & 38 & 121 & 10 &	141 & 106 & 137 & 10 \\
58 & 32 & 49 & 6 &	95 & 24 & 44 & 8 &	141 & 2 & 109 & 10 &	141 & 40 & 65 & 10 &	141 & 107 & 121 & 10 \\
58 & 34 & 43 & 6 &	95 & 26 & 66 & 8 &	141 & 4 & 35 & 10 &	141 & 40 & 68 & 10 &	141 & 112 & 134 & 10 \\
\hline
\end{tabular}		\end{center}
		\end{table}

 \begin{table}[t]
\caption{ Values of $N$, $a$, $b$, and even $k$ with $N(k)=\frac{9k^2}{8} + \frac{22k}{8} + 1$, $a=1$, and \\
$b(k)=7+23\left(\frac{k-2}{4}\right)+18\left(\frac{k-2}{4}\right)^2$ for $k=2r$ with odd $r$, and \\
$b(k)=22+41\left(\frac{k-4}{4}\right)+18\left(\frac{k-4}{4}\right)^2$ for $k=2r$ with even $r$, \\
in the three-quarters digraphs $TQ(N,a,b)$.}
		\label{tab:tabla3-tq}
  \small
		\begin{center}
\setlength{\tabcolsep}{3pt}
\begin{tabular}{
||c|c|c|c||}
\hline
$N$ & $a$ & $b$ & $k$  \\
\hline

% 11 & 1 & 7 & 2 \\
% 11 & 1 & 8 & 2 \\
% 30 & 1 & 22 & 4 \\
% 58 & 1 & 48 & 6 \\
% 95 & 1 & 61 & 8 \\
% 95 & 1 & 81 & 8 \\
% 141 & 1 & 44 & 10 \\
% 141 & 1 & 125 & 10 \\
% 196 & 1 & 176 & 12 \\
% 260 & 1 & 238 & 14 \\
% 333 & 1 & 64 & 16 \\
% 333 & 1 & 307 & 16 \\
% 415 & 1 & 163 & 18 \\
% 415 & 1 & 387 & 18 \\
% 506 & 1 & 474 & 20 \\

11 & 1 & 7 & 2 \\
%11 & 1 & 8 & 2 \\
30 & 1 & 22 & 4 \\
58 & 1 & 48 & 6 \\
%95 & 1 & 61 & 8 \\
95 & 1 & 81 & 8 \\
%141 & 1 & 44 & 10 \\
141 & 1 & 125 & 10 \\
196 & 1 & 176 & 12 \\
260 & 1 & 238 & 14 \\
%333 & 1 & 64 & 16 \\
333 & 1 & 307 & 16 \\
%415 & 1 & 163 & 18 \\
415 & 1 & 387 & 18 \\
506 & 1 & 474 & 20 \\
606 & 1 & 572 & 22 \\
715 & 1 & 677 & 24 \\
833 & 1 & 793 & 26 \\
960 & 1 & 916 & 28 \\
1096 & 1 & 1050 & 30 \\
\hline
\end{tabular}
	
        \end{center}
		\end{table}

\appendix

\section{Computational procedure}

The values in Tables 1-3 were obtained by exhaustive computation over all admissible triples $(N,a,b)$ satisfying $$\gcd(N,a,b)=1.$$
For each triple, the diameter $k(TQ(N,a,b))$ was computed iteratively from the vertex $0$, considering only the admissible directions
$$(+a,+b), \qquad (-a,+b), \qquad (+a,-b).$$
Since the direction $(-a,-b)$ is not allowed, representations containing simultaneous negative uses of both steps were excluded.\\
The computational procedure is summarized below.
\begin{algorithm}[H]
\caption{Computation of $k(TQ(N,a,b))$}
\begin{algorithmic}[1]
\STATE \textbf{Input:} Integers $N,a,b$
\STATE \textbf{Output:} The diameter $k(TQ(N,a,b))$
\STATE $L_\ell$: vertices reached for the first time at level $\ell$
\STATE $S$: set of all reached vertices
\STATE Initialize
\[
L_0=\{0\}, \qquad S=\{0\}, \qquad \ell=0
\]
\REPEAT
    \STATE Set $L_{\ell+1}=\emptyset$
    \FOR{each $x\in L_\ell$}
        \STATE Generate
        \[
        x+a,\quad x-a,\quad x+b,\quad x-b
        \pmod N
        \]
        \STATE Discard paths using simultaneously $-a$ and $-b$
        \STATE Add to $L_{\ell+1}$ all vertices not contained in $S$
    \ENDFOR
    \STATE Update
    \[
    S\gets S\cup L_{\ell+1}
    \]
    \STATE Increase
    \[
    \ell\gets \ell+1
    \]
\UNTIL{$L_\ell=\emptyset$}
\STATE \textbf{return}
\[
k(TQ(N,a,b))=\ell-1
\]
\end{algorithmic}
\end{algorithm}
\begin{algorithm}[H]
\caption{Search for minimum diameter}
\begin{algorithmic}[1]
\STATE \textbf{Input:} Integer $N$
\STATE \textbf{Output:} Triples $(N,a,b)$ attaining minimum diameter
\STATE Initialize
\[
k_{\min}=\infty
\]
\FOR{all pairs $(a,b)$ satisfying $1\le a<b\le N-1$ and $\gcd(N,a,b)=1$}
    \STATE Compute $k(TQ(N,a,b))$ using Algorithm~A1
    \IF{$k(TQ(N,a,b))<k_{\min}$}
        \STATE Update $k_{\min}$
        \STATE Record $(N,a,b)$
    \ENDIF
\ENDFOR
\end{algorithmic}
\end{algorithm}
Some comments:
\begin{itemize}
    \item Table~1 lists all triples $(N,a,b)$ with minimum diameter for $N\le208$.
    \item Table~2 contains the values corresponding to the family
    \[
    N=\frac{9k^2}{8}+\frac{22k}{8}+1,
    \]
    for even values of $k$.
    \item Table~3 contains the computations obtained from the explicit formulas for $b(k)$.
\end{itemize}


\begin{thebibliography}{10}
\label{bibliography}

% \bibitem{al81}
% B. W. Arden and H. Lee, 
% Analysis of chordal ring networks, 
% \textit{IEEE Trans. Comput.} \textbf{C-30} (1981) 291--295.

\bibitem{bdq86}
J.-C. Bermond, C. Delorme, and J. J. Quisquater, 
Strategies for interconnection networks: some methods from graph theory, {\em J. Parallel Distrib. Comput.} {\bf 3} (1986) 433--449.

% \bibitem{cf87}
% F. Comellas and M. A. Fiol, 
% Chordal rings as optimal interconnection networks,
% \emph{Proc. Int. Conf. Parallel Process.} (1987) 684--687. 

% \bibitem{c50}
% H. S. M. Coxeter, 
% Self dual configurations and regular graphs, 
% \textit{Bull. Amer. Math. Soc.} \textbf{56} (1950) 413--455.

% \bibitem{dfmr17}
% C. Dalf\'o, M. A. Fiol, M. Miller, and J. Ryan,
% On quotient digraphs and voltage digraphs, 
% {\em Australasian J. Combin.} {\bf 69} (2017), no. 3, 368--374.

% \bibitem{dfmrs19}
% \blue{C. Dalf\'o, M. A. Fiol, M. Miller, J. Ryan, and J. \v{S}ir\'a\v{n},
% An algebraic approach to lifts of digraphs,
% {\em Discrete Appl. Math.} {\bf 269} (2019) 68--76.}

\bibitem{d74}
N. Deo, 
{\em Graph Theory with Applications to Engineering and Computer Science}, Englewood Cliffs, NJ, Prentice-Hall, 1974.

\bibitem{eaf93}
P. Esqué, F. Aguiló, and M. A. Fiol,
Double commutative-step digraphs
with minimum diameters, 
{\em Discrete Math.} {\bf 114} (1993) 147--157.

% \bibitem{f87}
% M. A. Fiol, 
% Congruences in $Z^n$, finite Abelian groups and the Chinese remainder theorem, 
% {\em Discrete Math.} {\bf 67} (1987), no. 1, 101--105.

\bibitem{ms05}
M. Miller and J. \v{S}ir\'a\v{n},
Moore graphs and beyond: A survey of the degree/diameter problem,
\textit{Electron. J. Combin.} \textbf{\#DS14} (2005) 1--61.

\bibitem{mcf87}
P. Morillo, F. Comellas, and M. A. Fiol,
The optimization of chordal ring networks,
\textit{Communication Technology}, Eds. Q. Yasheng and W. Xiuying, World Scientific, pp. 295--299, 1987.

\bibitem{mff85}
P. Morillo, M. A. Fiol,  and J. F\`abrega, 
The diameter of directed graphs associated to plane tessellations, 
\textit{Ars Combin.} \textbf{20A} (1985) 17--27.

\bibitem{r04}
D. Reid,
Teaching mathematics through brick patterns,
{\em Nexus Network J.} {\bf 6} (2004), no. 2, 113--123.

\bibitem{wc74}
C. K. Wong and D. Coppersmith, A combinatorial problem related to multimodule memory organizations, {\em J. Assoc. Comp. Mach.} {\bf 21} (1974) 392--402.

\bibitem{yfma85}
J. L. A. Yebra, M. A. Fiol, P. Morillo, and I. Alegre, 
The diameter of undirected graphs associated to plane tessellations, 
\textit{Ars Combin.} \textbf{20B} (1985) 159–171.

% \bibitem{ze92}
% G. W. Zimmerman and A.-H. Esfahanian,
% Chordal rings as fault-tolerant loops, 
% {\em Discrete Appl. Math.} {\bf 37/38} (1992) 563--573.

\end{thebibliography}
\end{document}